\documentclass[12pt]{article}
\usepackage{amsmath, amssymb,amsthm}
\usepackage{graphicx}
\usepackage{amsthm}
\usepackage{a4wide}
\usepackage{amssymb}
\usepackage{pdfsync}
\usepackage{epstopdf}
\usepackage{enumerate}

\newtheorem{theorem}{Theorem}
\newtheorem{lemma}{Lemma}

\newtheorem{remark}{Remark}
\newtheorem{prop}{Proposition}

\newtheorem{fact}{Fact}
\newtheorem{conj}{Conjecture}

\newtheorem{cor}{Corollary}

\def\le{\leqslant}
\def\ge{\geqslant}

\begin{document}

\title{Tur\'an numbers for  3-uniform linear paths of length 3}

\author{Eliza Jackowska
\\A. Mickiewicz University\\Pozna\'n, Poland\\{\tt elijac@amu.edu.pl}
\and Joanna Polcyn
\\A. Mickiewicz University\\Pozna\'n, Poland\\{\tt joaska@amu.edu.pl}
 \and Andrzej Ruci\'nski
\thanks{Research supported by the Polish NSC grant xxxxxxx.
}
\\A. Mickiewicz University\\Pozna\'n, Poland\\{\tt rucinski@amu.edu.pl}
}

\date{\today}

\maketitle

\begin{abstract}
In this paper we confirm a conjecture of F\"uredi, Jiang, and Seiver, and determine an exact
formula for the Tur\'an number $ex_3(n; P_3^3)$ of the 3-uniform linear path $P^3_3$ of length 3,
valid for all $n$. It coincides with the analogous formula for the 3-uniform
triangle $C^3_3$, obtained earlier by Frankl and F\"uredi for $n\ge 75$ and Cs\'ak\'any  and Kahn for
all $n$. In view of this coincidence, we also determine a `conditional' Tur\'an number, defined as
the maximum number of edges in a $P^3_3$-free 3-uniform hypergraph on $n$ vertices which is \emph{not} $C^3_3$-free.
\end{abstract}

\section{Introduction}\label{intro}
A \emph{$k$-uniform hypergraph} (or  \emph{$k$-graph}, for short) is an ordered pair $H=(V,E)$, where $V$ is a finite set and $E\subseteq \binom Vk$ is a family of $k$-element subsets of $V$. We often identify $H$ with $E$, for instance, writing $|H|$ for the number of edges in $H$. Given a positive integer $n$ and a family of $k$-graphs $\mathcal F$, we say that a $k$-graph $H$ is \emph{$\mathcal F$-free} if $H$ contains no member of $\mathcal F$ as a subhypergraph. \emph{The
Tur\'an number} $ex_k(n; \mathcal F)$ is defined as the maximum number of edges in an $\mathcal F$-free $k$-graph on
$n$ vertices. We set $ex_3(0;\mathcal F)=0$ for convenience.

An $n$-vertex $k$-graph $H$ is called \emph{extremal} with respect to $\mathcal F$ if $H$ is $\mathcal F$-free and $|H|=ex_k(n;\mathcal F)$. We denote by $Ex_k(n;\mathcal F)$ the set of all, pairwise non-isomorphic $n$-vertex $k$-graphs which are extremal with respect to $\mathcal F$. If $\mathcal F=\{F\}$, then we write \emph{$F$-free} instead of \emph{$\{F\}$-free} and write $ex_k(n;  F)$, and  $Ex_k(n;  F)$ instead of $ex_k(n;\{F\})$ and $Ex_k(n;\{F\})$.

 A \emph{linear path}
$P^k_\ell$ (a.k.a. \emph{loose path}, though some authors mean by this term something else) is a $k$-graph with $\ell$ edges $e_1,\dots, e_{\ell}$ such that
$|e_i\cap e_j|=0$ if $|i-j|>1$ and $|e_i\cap e_j|=1$ if $|i-j|=1$ (see Fig.\ref{Fig1} for $P^3_3$).  F\"uredi, Jiang, and Seiver \cite{fjs} have determined $ex_k(n; P^k_\ell)$ for all $k\ge 4$,
$\ell\ge 1$, and sufficiently large $n$. In particular, their result for $\ell=3$ states that
$ex_k(n; P^k_3)=\binom{n-1}{k-1}$. They conjectured that this formula remains valid  in the case
$k=3$ too.
It is interesting to note that the case $k=3$, $\ell\ge 4$, has also been solved, but again for
large $n$ only (see~\cite{kmv}). So,  the sole remaining instance is  $k=\ell=3$ which we settle here for all $n$, confirming the above mentioned conjecture from
\cite{fjs}.

\bigskip
\begin{figure}[!ht]
\centering
\includegraphics [width=5cm]{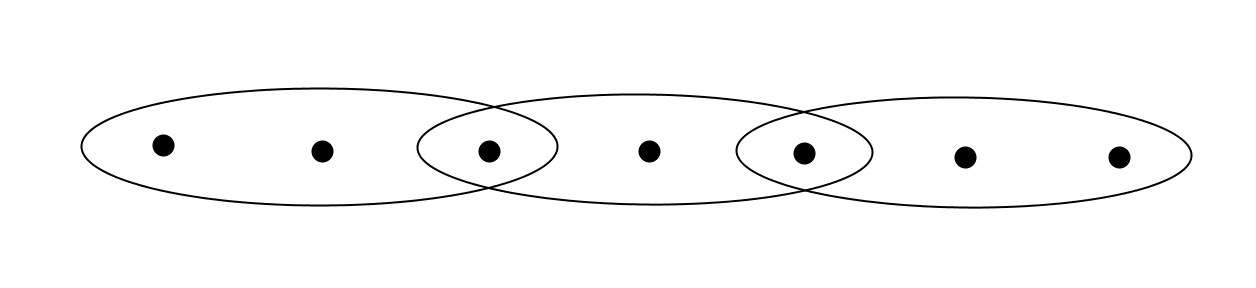}
\caption{The linear path  $P^3_3$}
\label{Fig1}
\end{figure}

 Let $K_n^k$ stand for \emph{the complete $k$-graph} with $n$ vertices, that is, one with $\binom nk$ edges. Note that when $n<k$ this is just a set of $n$ isolated vertices. \emph{A star}  is a hypergraph containing a vertex which belongs to all of its edges. An $n$-vertex $k$-uniform star with $\binom{n-1}{k-1}$ edges is called \emph{full} and denoted by $S^k_n$. By $F\cup H$  we denote the union of vertex disjoint copies of  $k$-graphs $F$ and $H$.

In this paper we prove two theorems. Our main theorem, Theorem \ref{main}, determines the Tur\'{a}n numbers for 3-uniform linear paths of length 3, for all $n$. Moreover, to each $n$ we match a unique extremal 3-graph.

\begin{theorem}\label{main}

 $$ex_3(n;P^3_3)=\left\{ \begin{array}{ll}
 \binom n3 & \textrm{ and $\quad Ex_3(n;P^3_3)=\{K^3_n\}\qquad\quad$\;\; for $n\le6,$ }\\
20 & \textrm{ and $\quad Ex_3(n;P^3_3)=\{K^3_6\cup K^3_1\}\quad$ for $n=7,$ }\\
\binom{n-1}{2} & \textrm{ and $\quad Ex_3(n;P^3_3)=\{S^3_n\}\qquad\quad$\;\; for $n\ge 8$.}
\end{array} \right.
$$

\end{theorem}
\bigskip

The proof of Theorem \ref{main} relies on a similar result for 3-uniform linear cycles, or triangles.
Let $C^3_3$ be \emph{the triangle} defined as a 3-graph on 6 vertices $a,b,c,d,e,f$
and with 3 edges $\{a,b,c\}$, $\{c,d,e\}$, and $\{e,f,a\}$.
It was proved  in \cite{ff} that $ex_3(n;C^3_3)=\binom{n-1}2$ for all
$n\ge 75$. This has been later extended  by Cs\'ak\'any
and Kahn \cite{ck} to cover all $n$.

\begin{theorem}[\cite{ff,ck}]\label{cycle} For $n\ge6$, $ex_3(n;C^3_3)=\binom{n-1}2$. Moreover, for $n\ge8$,
$Ex_3(n;C^3_3)=\{S^3_n\}$.
\end{theorem}

Theorem \ref{cycle} is the starting point of our proof of Theorem \ref{main}. Indeed, we
show that having a triangle in a 3-graph with at least $\binom{n-1}2$ edges leads to the existence of a copy of $P^3_3$. In
fact, it has turned out that the presence of $C^3_3$ pushes down the number of edges a $k$-graph may have without
containing a copy of $P_3^3$.  Motivated by this phenomenon, we also determine the largest number of edges in an $n$-vertex $P^3_3$-free 3-graph, $n\ge6$, which contains
a triangle.  We denote this `conditional' Tur\'an number by $ex_3(n;P^3_3|C^3_3)$ and the corresponding extremal family by $Ex_3(n;P^3_3|C^3_3)$. Our second
result expresses  $ex_3(n;P^3_3|C^3_3)$ in terms of the ordinary Tur\'an numbers $ex_3(n;P^3_3)$.

\begin{theorem}\label{second}
For $n\ge6$,
$$ex_3(n;P^3_3|C^3_3)=20+ex_3(n-6;P^3_3).$$
The only element of $Ex_3(n;P^3_3|C^3_3)$ is the disjoint union of $K_6^3$ and the unique extremal $P^3_3$-free 3-graph on
$n-6$ vertices.
\end{theorem}

Theorem \ref{second}, combined with Theorem \ref{main}, yields also the exact value of $ex_3(n;P^3_3|C^3_3)$. For brevity, we state it for $n\ge14$ only.

\begin{cor}
For $n\ge14$,
$$ex_3(n;P^3_3|C^3_3)=20+\binom{n-7}2\quad\mbox{ and }\quad Ex_3(n;P^3_3|C^3_3)=\{K_6^3\cup S_{n-6}^3\}.$$
\end{cor}
Our last result follows rather from the proof of Theorem \ref{second} than from the theorem itself. Let $ex_3^{con}(n;P^3_3|C^3_3)$ be defined as $ex_3(n;P^3_3|C^3_3)$, but  for \emph{connected} graphs only.

\begin{cor}\label{connected}
For $n\ge9$, $$ex^{con}_3(n;P^3_3|C^3_3)=3n-8.$$
\end{cor}

\begin{remark}[{\bf Disjoint unions of $P_3^3$}]\rm For a positive integer $s$, let $sF$ denote the vertex-disjoint union of $s$ copies of a hypergraph $F$.
Bushaw and Kettle \cite{bk} determined, for large $n$, the Tur\'an number $ex_k(n;sP_\ell^k)$, but only for those instances for which the Tur\'an number $ex_k(n;P_\ell^k)$ had been known (they used induction on $s$).
In particular, they have shown, for large $n$, that \emph{if} $ex_3(n;P_3^3)=\binom n3-\binom{n-1}3$, then $ex_3(n;sP_3^3)=\binom n3-\binom{n-2s+1}3$, providing also the unique extremal 3-graph, which happens to be the same as that for $M^3_{2s}$, the matching of size $2s$ (see \cite{Erdos}). By proving Theorem \ref{main}, we have, at the same time, verified the latter formula unconditionally.
\end{remark}

\section{Preliminaries}\label{pre}

In what follows $H$ is always a $P_3^3$-free 3-graph with $V(H)=V$ and $|V|=n\ge 7$, containing a copy $C$ of the triangle $C_3^3$.
 Let
 $$U=V(C),\quad U=U_1\cup U_2,\quad \mbox{ where }\quad U_1=\{y_1,y_2,y_3\},\quad U_2=\{x_1,x_2,x_3\},$$
 and
 $$ C=\{\{x_i,y_j,x_k\}:\; \{i,j,k\}=\{1,2,3\}\},$$
 so that for $i=\{1,2\}$, $U_i$ is the set of vertices of degree $i$ in $C$ (see Fig. \ref{Fig2}).

\bigskip
\begin{figure}[!ht]
\centering
\includegraphics [width=5cm]{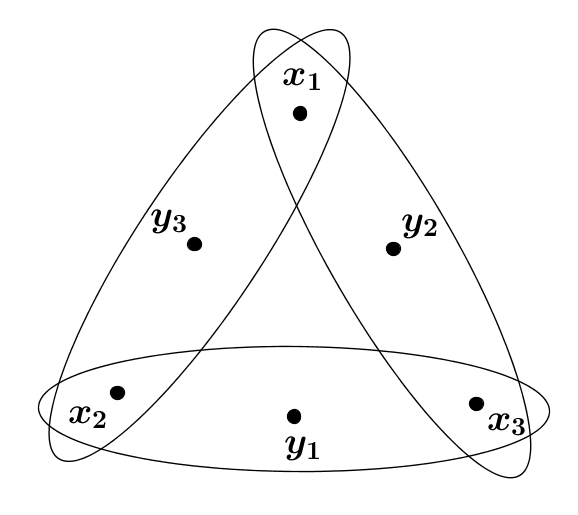}
\caption{The triangle $C^3_3$}
\label{Fig2}
\end{figure}

Further, let
  $$W=V\setminus U=\{w_1,\dots, w_s\},\qquad |W|=s=n-6.$$ We split the set of  edges of $H$  into three subsets (see Fig. \ref{Fig3}),
  $$H[U]=H\cap \tbinom{U}{3},\quad H[W]=H\cap \tbinom{W}{3}\quad\mbox{ and }\quad H(U,W)=H\setminus (H[U] \cup H[W]).$$
\bigskip
\begin{figure}[!ht]
\centering
\includegraphics [width=10cm]{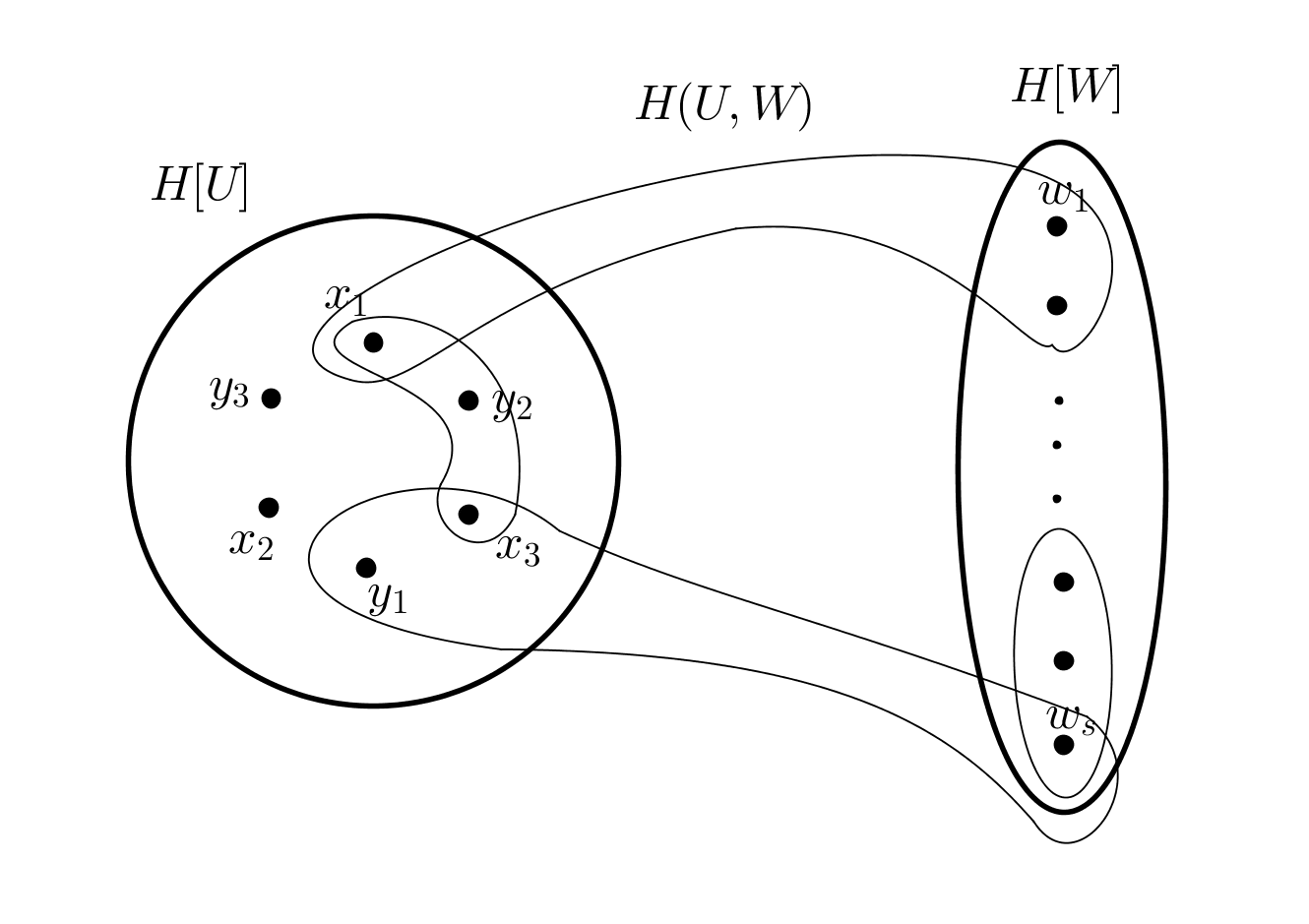}
\caption{The partition of the set of edges of $H$}
\label{Fig3}
\end{figure}
Let us also define two sets of triples (which are not necessarily edges of $H$):
$$T_1=\left\{\{x_i,y_i,w_l\}:\; 1\le i\le 3,\; 1\le l\le s\right\},\qquad T_2=\left\{\{x_i,x_j,w_l\}:\;1\le i<j\le 3,\; 1\le l\le s\right\}$$
(see Fig. \ref{Fig4}) and set
$$T=T_1\cup T_2.$$
\bigskip
\begin{figure}[!ht]
\centering
\includegraphics [width=15cm]{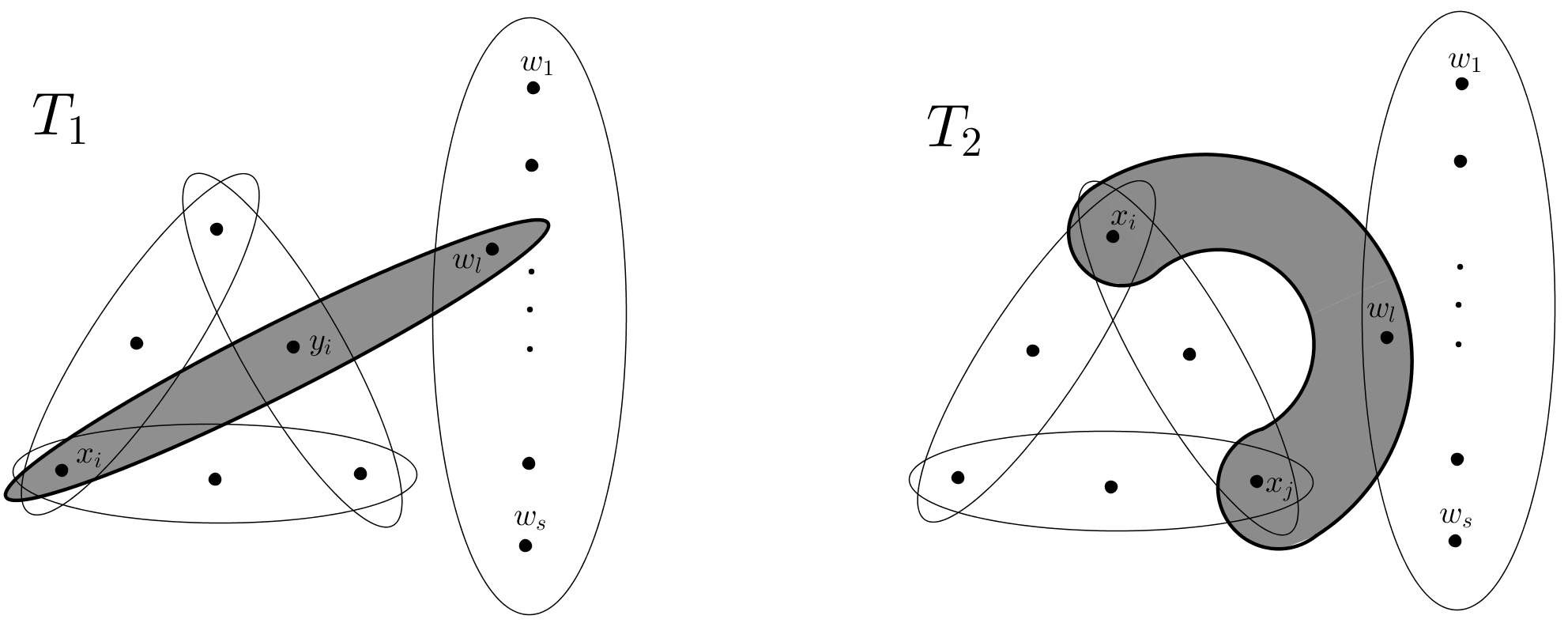}
\caption{The edges in sets $T_1$ and $T_2$ are shaded}
\label{Fig4}
\end{figure}

We begin with several simple observations all of which can be  verified by inspection.
 The first three have been already made in \cite{j}. First of them says that although, in principle, $H(U,W)$ may consist of edges having one vertex in $U$ (and two  in $W$), the assumption that $H$ is $P^3_3$-free makes it impossible. For the same reason, out of the potential edges with two vertices in $U$ (and one in $W$), only those listed in $T$ can actually occur in $H$.
\begin{fact}[\cite{j}, Facts 1-3]\label{euw}
     $$H(U,W)=H\cap T.$$
\end{fact}
The next observation excludes coexistence in $H$ of two edges, one from $T$ and the other from $\tbinom W3$, if they share a common vertex. (see Fig. \ref{Fig5}
).

\bigskip
\begin{figure}[!ht]
\centering
\includegraphics [width=15cm]{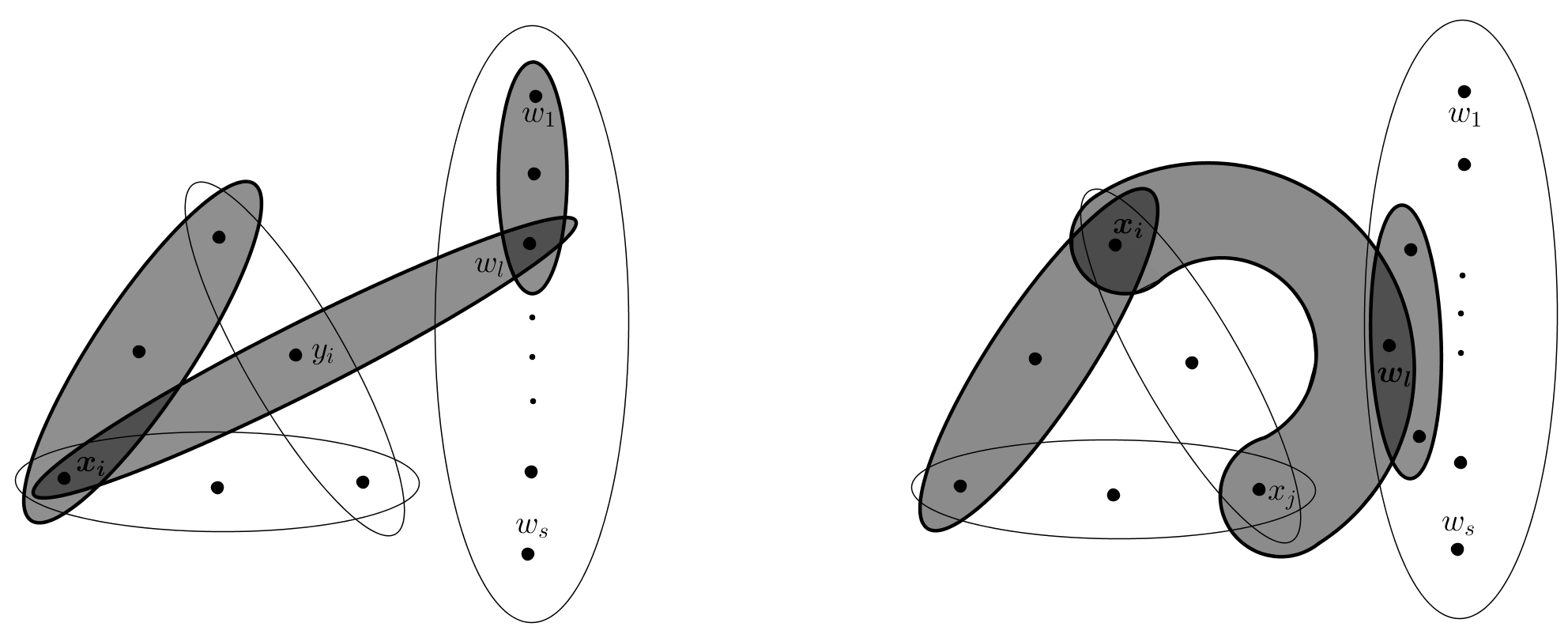}
\caption{Illustration of Fact \ref{pusto}}
\label{Fig5}
\end{figure}

\begin{fact}[\cite{j}, Fact 6]\label{pusto}
    If $e\in T$, $g\in \binom W3$, and $e\cap g\neq \emptyset$, then $C\cup\{e\}\cup\{g\}\supset P^3_3$.
\end{fact}
 Similarly,  coexistence is impossible between any two disjoint edges, one from $T_1$ and the other from $T$ (see Fig. \ref{Fig6}).

\bigskip
\begin{figure}[!ht]
\centering
\includegraphics [width=15cm]{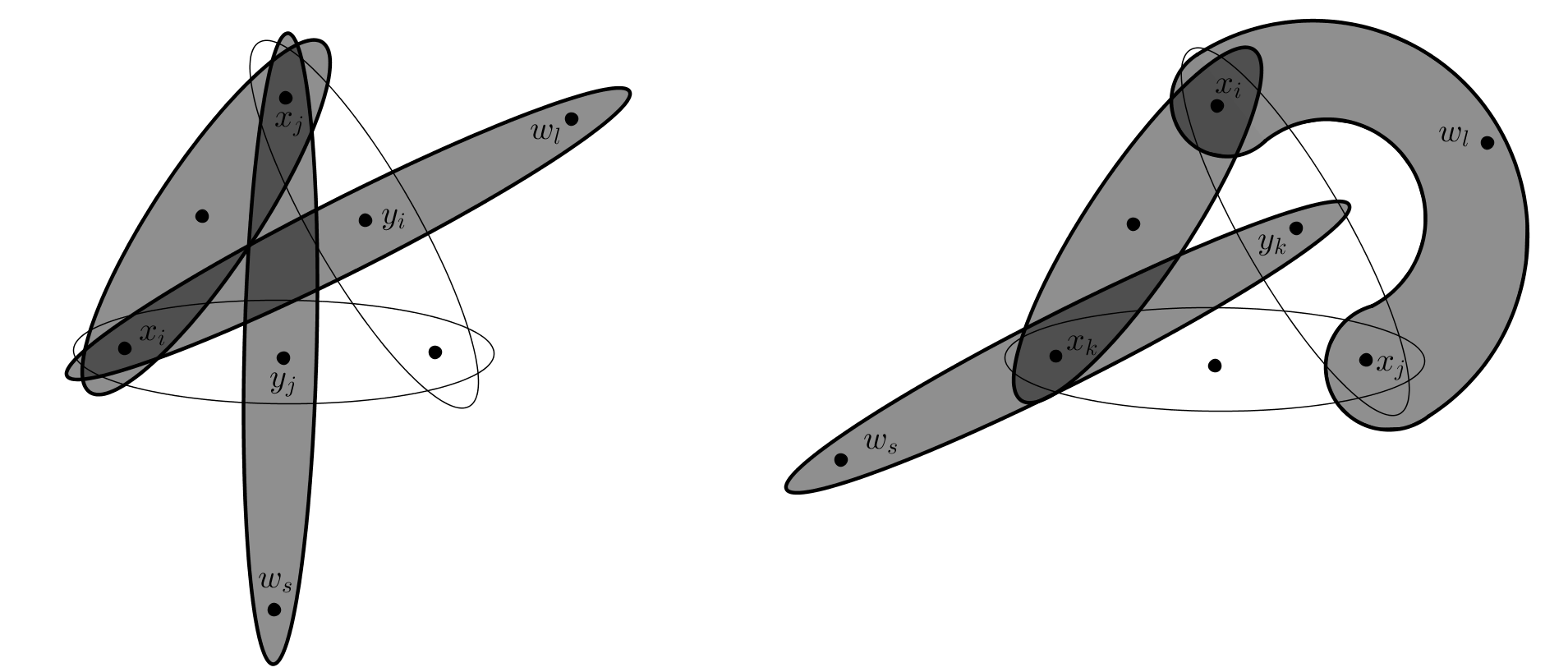}
\caption{Illustration of Fact \ref{disjoint}}
\label{Fig6}
\end{figure}

\begin{fact}\label{disjoint}
     If $e\in T_1$, $f\in T$, and $e\cap f=\emptyset$, then $C\cup\{e\}\cup\{f\}\supset P^3_3$.
\end{fact}

We will also need the following simple consequence of the K\"onig Theorem.

 \begin{fact}\label{konik}
    In a $t\times s$ bipartite graph, where $t\le s$, the largest possible number of edges not producing a matching of size $m+1$, $m\le t$, is $sm$.
 \end{fact}

 Combining Fact \ref{disjoint} for $e,f\in T_1$ with Fact \ref{konik}, we obtain the following corollary.

\begin{cor}\label{HT1}
    For $s\ge3$,
\begin{equation}\label{ttt}
|H\cap T_1|\le s.
\end{equation}
\end{cor}
\begin{proof}
Let $B$ be the  auxiliary $3\times s$ bipartite graph with vertex classes $\{1,2,3\}$ and $W$, where $\{i,w\}$ is an edge of $B$ if $\{x_i,y_i,w\}\in H$.
Thus, $|B|=|H\cap T_1|$.
By Fact \ref{disjoint}, there are no disjoint edges in $B$. Hence, by Fact \ref{konik} with $t=3$ and $m=1$, $|B|\le s$.
\end{proof}
Another consequence of Fact \ref{disjoint} has been already proved in \cite{j}. We  reproduce that proof for the sake of self-containment.

\begin{prop}[\cite{j}, Fact 4]\label{3s}
    For $s\ge 2$, $$|H\cap T| \le 3s.$$
\end{prop}

\begin{proof}
     We have
\begin{equation}\label{T=}
|T_1|=|T_2|=3s.
\end{equation} Construct an auxiliary bipartite graph $B=(T_1, T_2; \mathcal{E})$, where $\{e,f\}\in \mathcal{E}$ if $e\cap f=\emptyset$. It follows from Fact \ref{disjoint} that if $\{e, f\}\in \mathcal{E}$, then $|\{e, f\}\cap H|\le 1$. Observe also that the graph $B$ is $(s-1)$-regular. Thus, by Hall's theorem, it has a perfect matching $M$. As at most one edge of each pair $\{e,f\}\in M$ can be  in $H$, we infer that $|H\cap T|\le 3s$.
\end{proof}

\section{The lemmas}

To prove Theorem \ref{main}, we will need the following lemma which, with the notation of Section \ref{pre}, puts a cap on the total number of edges in the subgraphs $H[U]$ and $H(U,W)$, provided the latter is nonempty.

\begin{lemma}\label{eueuw1}
    For $s \ge 1$, if $H(U,W)\neq \emptyset$, then
    $$|H[U]|+|H(U,W)|\le 13+\max\{3s,6\}.$$
\end{lemma}

\begin{proof}

We begin by deducing  upper bounds on $|H[U]|$ implied by the presence of an edge in
$$H(U,W)=(H\cap T_1)\cup(H\cap T_2).$$
Assume first that $H\cap T_1\neq \emptyset$, say $\{x_1,y_1,w\}\in H\cap T_1$ for some $w\in W$. Let (cf. Fig. \ref{Fig2})
      $$X_1=\left\{\{x_1,y_2,y_3\},\{x_2,y_2,y_3\},\{x_3,y_2,y_3\},\{x_2,y_1,y_3\},\{x_3,y_1,y_2\},\{x_2,x_3,y_2\},\{x_2,x_3,y_3\}\right\}.$$
          One can easily check that if $H\cap X_1\neq \emptyset$, then $P^3_3\subseteq H$, a contradiction. Hence,
     $H[U] \subseteq \binom{U}{3}\setminus X_1,$
     and so,
    \begin{equation}\label{t1}|H[U]|\le\left|\tbinom{U}{3}\right|-|X_1|=20-7=13.\end{equation}
Similarly, if $e=\{x_1,x_2,w\}\in H\cap T_2$, then, by considering the set
 $$X_2=\{\{y_1,y_2,y_3\},\{x_2,y_1,y_3\},\{x_3,y_1,y_3\},\{x_1,y_2,y_3\},\{x_3,y_2,y_3\}\},$$
 one can show that
\begin{equation}\label{t2}|H[U]|\le\left|\tbinom{U}{3}\right|-|X_2|=20-5= 15.
\end{equation}
In summary,
\begin{equation}\label{t}   H(U,W)\neq\emptyset\Longrightarrow|H[U]|\le 15.
\end{equation}
Therefore, if $|H(U,W)| \le s$, then, with some margin,
    $$|H[U]|+|H(U,W)|\le 15 + s < 13+\max\{3s,6\}.$$
Consider now the case $|H(U,W)| > s$. Since by  Fact \ref{euw}, Proposition \ref{3s}, and (\ref{T=}), for all $s\ge 1$ we have
\begin{equation}\label{3s6}
|H(U,W)| \le \max\{3s,6\},
\end{equation}
it remains to show that (\ref{t1}) still holds. As explained above, this is the case when $H\cap T_1\neq\emptyset$. Otherwise, $|H\cap T_2|=|H(U,W)|>s$, and,  since $|W|=s$, we infer that there exists a vertex $w \in W$ and two edges $e,f\in H\cap T_2$, both containing $w$.  Then, necessarily, $|e\cap f\cap U| = 1$.
Say, $e\cap f\cap U = \{x_1\}$ (see Fig. \ref{Fig7}). Consequently,  to avoid a copy of $P_3^3$ in $H$,
we must have $H\cap Y = \emptyset$, where
$$Y = X_2\cup \{\{x_2,y_2,y_3\},\{x_2,y_1,y_2\},\{x_3,y_1,y_2\}\},$$
    and so,
$$|H[U]|\le \left|\tbinom{U}{3}\right|- |Y| =20 -8=12,$$
which is even better than (\ref{t1}).
In conclusion, for all $s\ge1$,

\begin{equation}\label{2s}
|H(U,W)| > s\Longrightarrow|H[U]|\le 13.
\end{equation}
\bigskip
\begin{figure}[!ht]
\centering
\includegraphics [width=8cm]{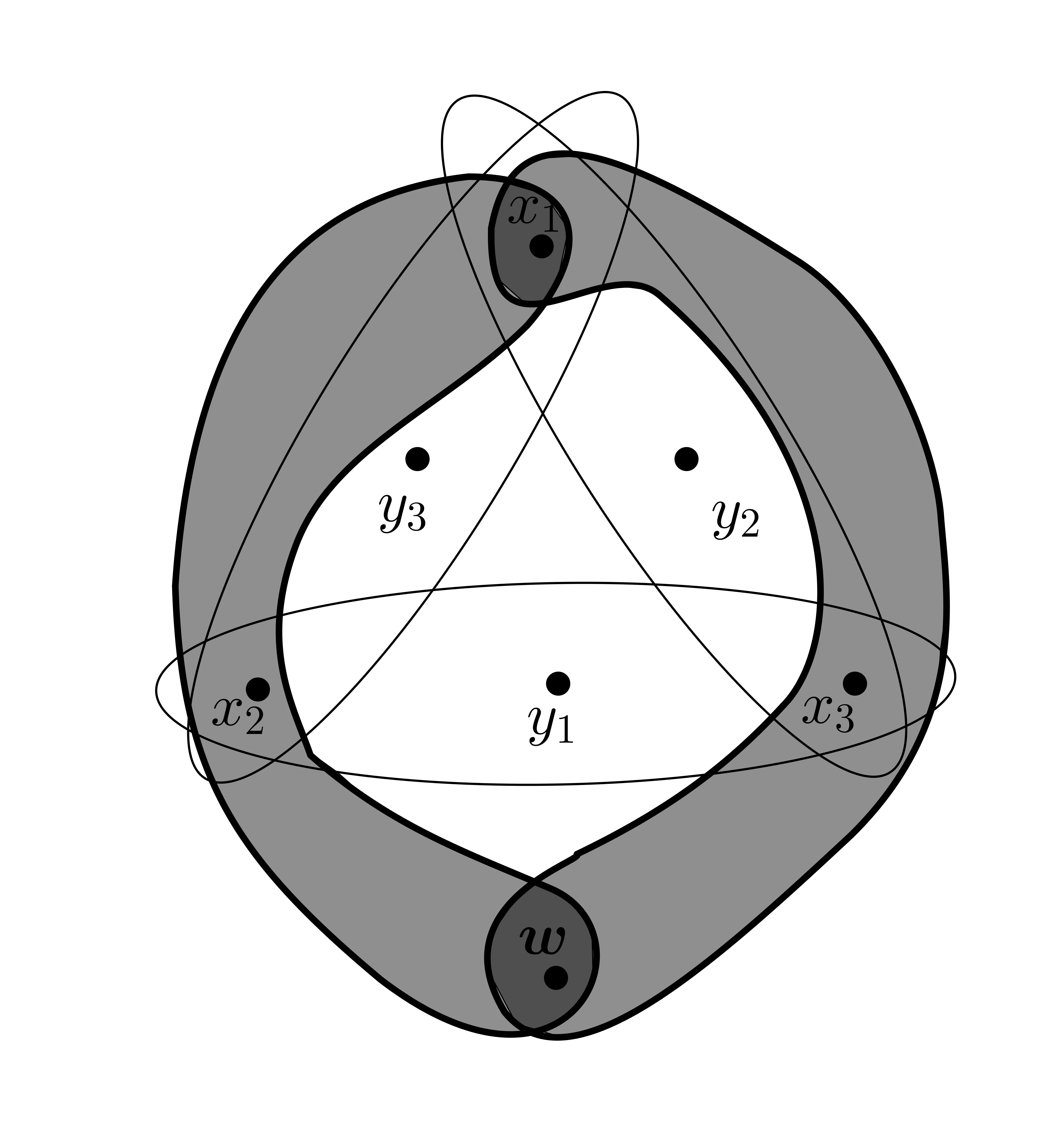}
\caption{Illustration to the proof of Lemma \ref{eueuw1}}
\label{Fig7}
\end{figure}
Putting together bounds (\ref{2s}) and (\ref{3s6}) completes the proof of Lemma \ref{eueuw1}.
\end{proof}

\bigskip

Since for $s\ge2$ we have $\max\{3s,6\}=3s$ and $|H[U]|\le\left|\binom{U}3\right|= 20\le14+3s$, Lemma
\ref{eueuw1} has the following immediate consequence, true no matter whether $H(U,W)=\emptyset$ or not.
\begin{cor}\label{eueuw}
    For $s \ge 2$,
    $$|H[U]|+|H(U,W)|\le 14+ 3s.$$
\end{cor}


In the proof of Theorem \ref{second} we will  need a further improvement, under additional constraints, of the bound in Corollary \ref{eueuw}.

\begin{lemma}\label{113s}
    For $s\ge 3$, if $H(U,W)\neq\emptyset$, then
     $$|H[U]| + |H(U,W)| \le 10 + 3s.$$
\end{lemma}

\begin{proof}
If $0 <|H(U,W)| \le s$, then, by (\ref{t}),
     $$|H[U]| + |H(U,W)| \le 15 + s < 10 + 3s.$$
Also, if $s < |H(U,W)| \le 2s$, then by (\ref{2s}), $$|H[U]| + |H(U,W)| \le 13 + 2s \le 10 +3s.$$

For the rest of the proof we are assuming that
\begin{equation}\label{7}|H(U,W)|=|H\cap T_1|+|H\cap T_2| \ge 2s+1\ge7.
\end{equation}
 We are going to show that
\begin{equation}\label{3t2}
|H[U]|\le 10.
\end{equation}
Then the lemma will follow by Proposition \ref{3s}.

Consider first the case when $H\cap T_1= \emptyset$. Then $|H\cap T_2| \ge 2s+1$ and, thus, there must exist a vertex $w\in W$ such that all three edges $\{x_i,x_j,w\}$, $1\le i < j \le 3$, belong to $H$ (see Fig. \ref{Fig8}).
\bigskip
\begin{figure}[!ht]
\centering
\includegraphics [width=8cm]{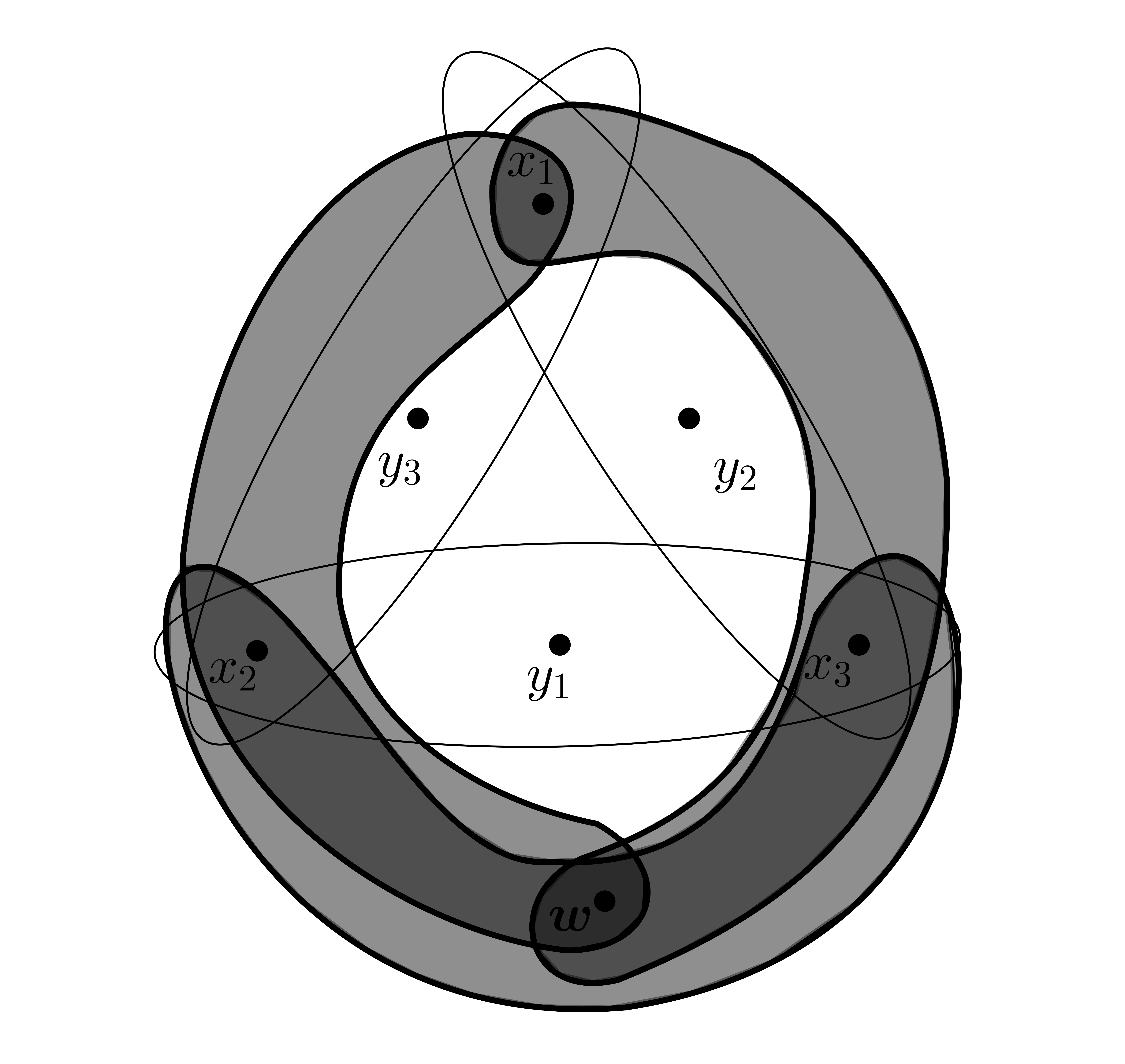}
\caption{Illustration to the proof of Lemma \ref{113s}: case $H\cap T_1= \emptyset$}
\label{Fig8}
\end{figure}
But then, since $H$ is $P_3^3$-free, we have $H\cap Z_1 = \emptyset$, where
 $$Z_1 = \{\{y_1,y_2,y_3\}, \{y_i,y_j,x_k\}:1\le i < j \le 3, 1 \le k \le 3\},\,\,\,\, |Z_1|=10.$$
Thus, (\ref{3t2})  holds.

 Assume now that $H\cap T_1\neq \emptyset$. W.l.o.g., let $h'=\{x_1,y_1,w'\}\in H$, where $w'\in W$, and distinguish two subcases.

\bigskip

 \noindent{\bf Subcase 1:} For some $w''\in W$, $w''\neq w'$, we have $h''=\{x_1,y_1,w''\}\in H$.
   By  Fact \ref{disjoint}, every edge of $H\cap T_2$ must intersect both, $h'$ and $h''$. Thus, every edge of $H\cap T_2$ contains vertex $x_1$. Since, by (\ref{ttt}),  $|H\cap T_1| \le s$, we infer that $|H \cap T_2| > s$. Consequently, there  exists a vertex $w\in W$ with $\{x_1,x_2,w\}$ and $\{x_1,x_3,w\}$ belonging to $H$ (see Fig. \ref{Fig9} for the case when $w=w'$).
\bigskip
\begin{figure}[!ht]
\centering
\includegraphics [width=8cm]{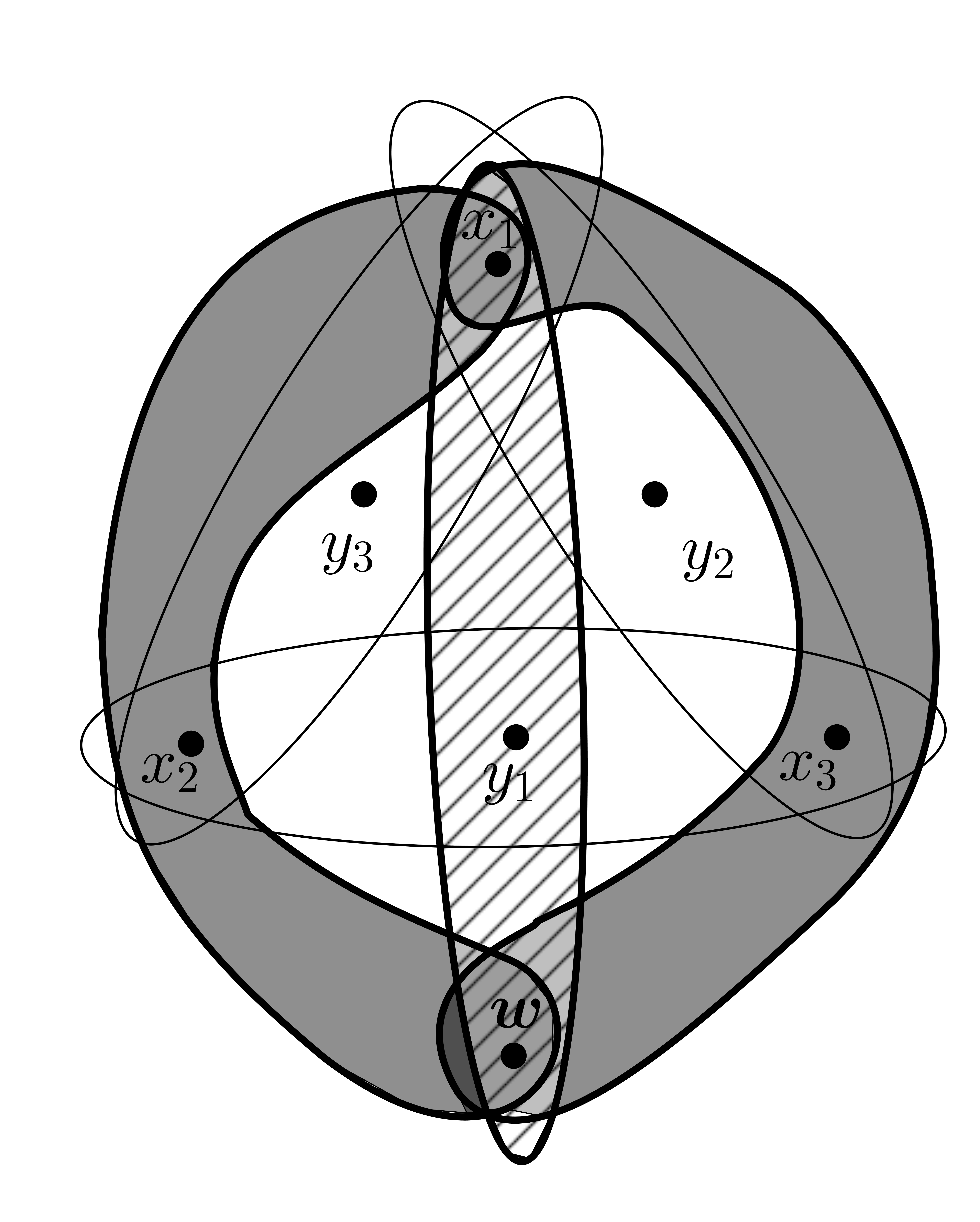}
\caption{Illustration to the proof of Lemma \ref{113s}: case $H\cap T_1\neq \emptyset$}
\label{Fig9}
\end{figure}
But then
$H\cap Z_2 = \emptyset$, where
$$Z_2= Y \cup X_1 = Y \cup \{\{x_2,x_3,y_2\},\{x_2,x_3,y_3\}\},\,\,\,\, |Z_2|=10,$$
and (\ref{3t2}) holds.

\bigskip

 \noindent{\bf Subcase 2:} $H\cap T_1\subseteq\{\{x_i,y_i,w'\},1\le i \le 3\}$. Set $|H\cap T_1|=t$, $1\le t \le 3$.
  By Fact \ref{disjoint}, for every $i=1,2,3$, if $\{x_i,y_i,w'\}\in H$ then $\{x_j,x_k,w\}\not\in H$ for all $w\neq w'$, where $\{j,k\}=\{1,2,3\}\setminus\{i\}$. Hence,
 
$$
|H\cap T_2|\le  t+(3-t)s,
$$
 and we have $t\le 2$ by (\ref{7}). Moreover, for $t=2$, 
 $2s-1\le |H\cap T_2|\le 2+s$
 which forces $s=3$, and, consequently, 
$|H\cap T_2|=5$.
  This, in turn, implies the existence in $H$ of all three edges $\{x_i,x_j,w'\}$, $1\le i < j \le 3$, as in the case $H\cap T_1=\emptyset$ discussed above, and, again (\ref{3t2}) holds. Finally, if $t=1$, that  is, $H\cap T_1=\{h'\}$, then, letting $e'=\{x_2,x_3,w'\}$,
   $$|(H\cap (T_2\setminus\{e'\})|\ge|H\cap T_2|-1=|H\cap T|-2\ge2s-1>s.$$
    Consequently, there exists a vertex $w\in W$ belonging to two edges of $T_2\setminus\{e'\}$. This means that regardless of whether $w=w'$ or $w\neq w'$, the edges $\{x_1,x_2,w\}$ and $\{x_1,x_3,w\}$ both belong to $H$. As this is the same configuration as  in Subcase 1 (cf. Fig. \ref{Fig9}), the bound (\ref{3t2}) holds again.


\end{proof}

\section{Proofs of Theorems \ref{main} and \ref{second}}\label{pt1}

\subsection{Proof of Theorem \ref{main}}
This proof is by induction on $n$. Since $P^3_3$ contains 7 vertices,  Theorem \ref{main} is
trivially true for $n\le 6$. Although we begin the inductive step at $n=8$ only, our proof has
the same logical structure for all $n\ge7$.
 First note that both candidates for the extremal 3-graph, $H_7:=K_6^3\cup K_1$ for $n=7$ and $H_n:=S_n^3$
 for $n\ge8$, are $P^3_3$-free.
 We will be assuming that $H$ is a $P_3^3$-free 3-graph, with $|V|=n$, $|H|\ge |H_n|$
 and $H\neq  H_n$. By Theorem \ref{cycle},
$H$ is going to contain a copy $C$ of the triangle $C_3^3$. From that point on we will make our way
toward an application of Lemma \ref{eueuw1}, leading to the inequality $|H|<|H_n|$,
contradicting our assumption. Ultimately, we will show that no $P^3_3$-free 3-graph on $n$ vertices and at least
$|H_n|$ edges exists, except for $H_n$ itself, which is precisely the statement of
Theorem~\ref{main}. Now come the details. Throughout, we keep the notation introduced  in Section
\ref{pre}.

\bigskip

\noindent {\bf $\mathbf{n=7}$ (initial step).} Let $H$ be a $P_3^3$-free 3-graph with $V(H)=V$,
$|V|=n=7$ (thus, $s=1$), $|H| \ge 20$, and let $H\neq K_6^3\cup K_1$. Note that $20 >
\binom{7-1}{2}=15$ and so, by Theorem \ref{cycle}, $H$ contains a copy $C$ of the triangle $C_3^3$.
As  $H\neq K_6^3\cup K_1$, we infer that $H(U,W)\neq\emptyset$. Hence, by Lemma \ref{eueuw1},
\begin{equation*}
|H[U]|+|H(U,W)|\le 13 + \max\{3s,6\} = 19 < 20,
\end{equation*}
a contradiction.

\bigskip

\noindent {\bf $\mathbf{n\ge8}$ (inductive step).} Let $H$ be a $P_3^3$-free 3-graph with $V(H)=V$,
$|V|=n\ge 8$, $|H| \ge \binom{n-1}{2}$ and let $H\neq S^3_n$. By Theorem \ref{cycle}, $H$ contains
a copy $C$ of the triangle $C_3^3$.  By Corollary \ref{eueuw}, with $s=n-6$, we get
$$|H|=|H[U]|+|H(U,W)|+|H[W]|\le 14 + 3s + ex_3(s;P_3^3).$$
Consequently, to complete the proof  it remains to show that
$$14 + 3s + ex_3(s;P^3_3) < \binom{n-1}{2} = \binom{s+5}{2},$$
that is, to show that
$$ex_3(s;P^3_3) <\binom{s+5}{2}-3s-14=\binom{s+2}{2} -5.$$
To this end, we rely on our induction's assumption, in particular, on the formula for $ex_3(s;P^3_3)$.
For $s=\{2,3,4,5,6\}$ (equivalently, $n=\{8,9,10,11,12\}$),  one can check by direct substitution that
$$ex_3(s;P_3^3)=\binom{s}{3} <\binom{s + 2}{2} - 5.$$
For $s=7$  ($n=13$), $$ex_3(s;P_3^3)=20 < \binom{7+2}{2} - 5 = 31.$$
Finally, for $s\ge 8$ ($n\ge 14$),
 $$ex_3(s;P_3^3)=\binom{s-1}{2} <\binom{s-1}{2} + 3s - 5 = \binom{s+2}{2} - 5.$$
\qed

\subsection{Proof of Theorem \ref{second}}\label{pt3}
Although not inductive, this proof is based on similar ideas to those used in the proof of Theorem \ref{main}, as well as on Theorem~\ref{main} itself.
There is nothing to prove for $n=6$.  From now on we will be assuming that $n\ge 7$, or equivalently, that $s\ge 1$
(again, we keep the notation introduced  in Section \ref{pre}).

Let $H$ be a $P_3^3$-free 3-graph with $V(H)=V$, $|V|=n\ge 7$, containing a copy $C$ of the
triangle $C^3_3$. Observe that if $H(U,W)=\emptyset$, then the only $P^3_3$-free, $n$-vertex 3-graph
with at least $20 + ex_3(n-6;P_3^3)$ edges consists of a  copy of $K_6^3$ and a $P^3_3$-free
extremal 3-graph on $n-6$ vertices. Consequently, in order to prove Theorem \ref{second}, it is
sufficient to show that if $H(U,W)\neq\emptyset$ then
$$|H| < 20 + ex_3(n-6;P_3^3).$$

Assume that $H(U,W)\neq\emptyset$. We split the set of vertices $W$ into two subsets:
$$W_1 = \{w\in W: \text{ there exists an edge } e\in H(U,W) \text{ such that } w\in e\},$$
and
$$W_2 = W\setminus W_1.$$
Set $|W_i|=s_i$, $i=1,2$, where $s_1+s_2=s=n-6$.
By Facts \ref{euw} and  \ref{pusto}, $H[W] \subset \tbinom{W_2}{3}$ (see Fig.\ref{Fig11}). It
turns out that all we need to show is that
$$|H[U]| + |H(U,W)| < 20 + ex_3(s_1;P^3_3).$$
\bigskip
\begin{figure}[!ht]
\centering
\includegraphics [width=10cm]{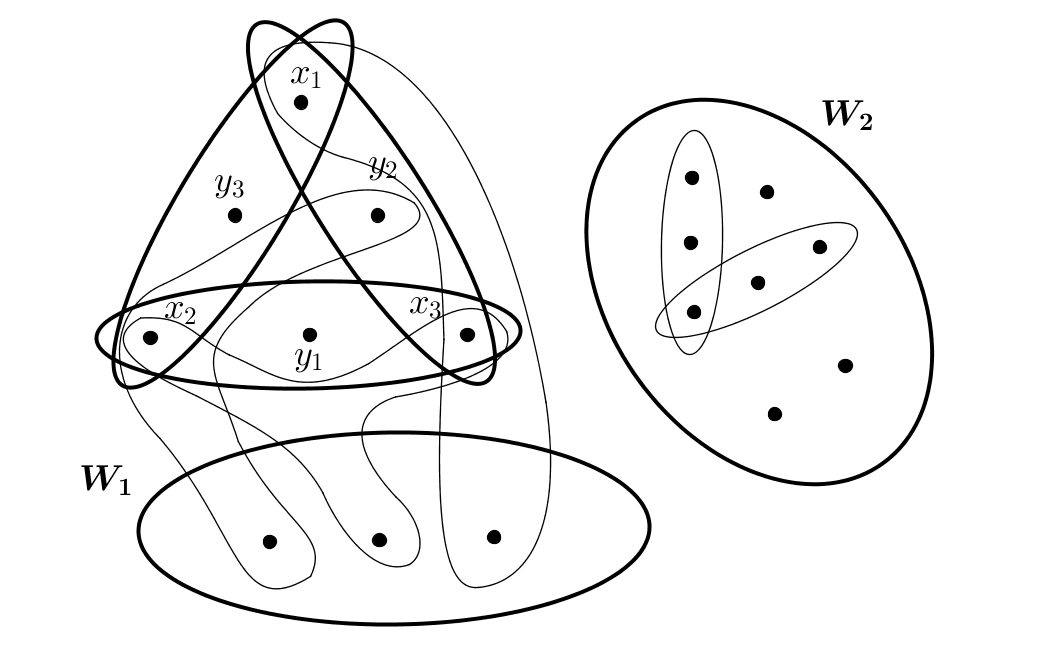}
\caption{The division of the set $W$ into two subsets $W_1$ and $W_2$}
\label{Fig11}
\end{figure}
Indeed, by the subadditivity of $ex_3(t;F)$ as a function of $t$, we will then have
$$|H|=|H[U]| + |H(U,W)| + |H[W]| <20 + ex_3(s_1;P^3_3)  + ex_3(s_2;P_3^3) \le 20 +ex_3(s; P^3_3).$$
For $1\le s_1\le 2$, we apply Lemma \ref{eueuw1} to the induced subhypergraph $H[U\cup W_1]$
to get
$$
|H[U]|+|H(U,W)| \le 13 + \max\{3s_1,6\}=19 < 20=20 + ex_3(s_1;P^3_3).
$$
Finally, assume that $s_1\ge 3$. By Lemma \ref{113s} applied to
$H[U\cup W_1]$ and by Theorem \ref{main} with $n:=s_1$, we conclude that
 $$|H[U]| + |H(U,W)| \le 10 + 3s_1< 20 + ex_3(s_1;P_3^3),$$
 where the verification of the last inequality is left to the reader. \qed

\begin{proof}[Proof of  Corollary \ref{connected}]

With the notation of the proof of Theorem \ref{second}, observe that the connectivity assumption implies that $W_2=\emptyset$. Thus, by Lemma \ref{113s}
$$|H|=|H[U]| + |H(U,W)| \le 10 + 3(n-6)=3n-8.$$
Moreover, the $3$-graph with vertex set $V$ and the edge set $\left(\binom U3\setminus Z_1\right)\cup T_2$ contains $C^3_3$, is $P^3_3$-free and has  $3n-8$ edges. \end{proof}

\section{Conditional Tur\'an numbers}
Inspired by Theorem \ref{second}, in this final section we discuss some restricted versions of
 Tur\'an numbers. We begin with a general definition of the conditional Tur\'an numbers.

Given an integer $n$, a family of $k$-graphs $\mathcal F$, and a family of $\mathcal F$-free $k$-graphs $\mathcal G$,   let
$ex_k(n;\mathcal F|\mathcal G)$ be the largest number of edges in an $n$-vertex $\mathcal F$-free $k$-graph $H$ such
that $H\supseteq G$ for some $G\in\mathcal G$. If $\mathcal F=\{F\}$ or $\mathcal G=\{G\}$, we will simply write
$ex_k(n;F| \mathcal G)$, $ex_k(n;\mathcal F|G)$, or $ex_k(n;F|G)$, respectively.

Of course, we have $ex_k(n;\mathcal F|\mathcal G)\le ex_k(n;\mathcal F)$.
For instance, comparing Theorems \ref{main} and \ref{second}, we see that for $n\ge14$
$$ ex_3(n;P^3_3)-ex_3(n;P^3_3|C_3^3)=6n-47.$$
In view of the equality $ex_3(n;P^3_3)=ex_3(n;C^3_3)$ (for $n\ge8$), it would be also interesting to calculate the reverse conditional Tur\'an number, namely $ex_3(n;C^3_3|P_3^3)$. For $n\ge7$, consider a 3-graph $H(n;C|P)$ consisting of an edge $\{x,y,z\}$ and all edges of the form $\{x,y,w\}$, $w\neq z$, and $\{z,w',w''\}$, where $\{w',w''\}\cap\{x,y\}=\emptyset$ (see Fig. \ref{Fig12}).

\bigskip
\begin{figure}[!ht]
\centering
\includegraphics [width=9cm]{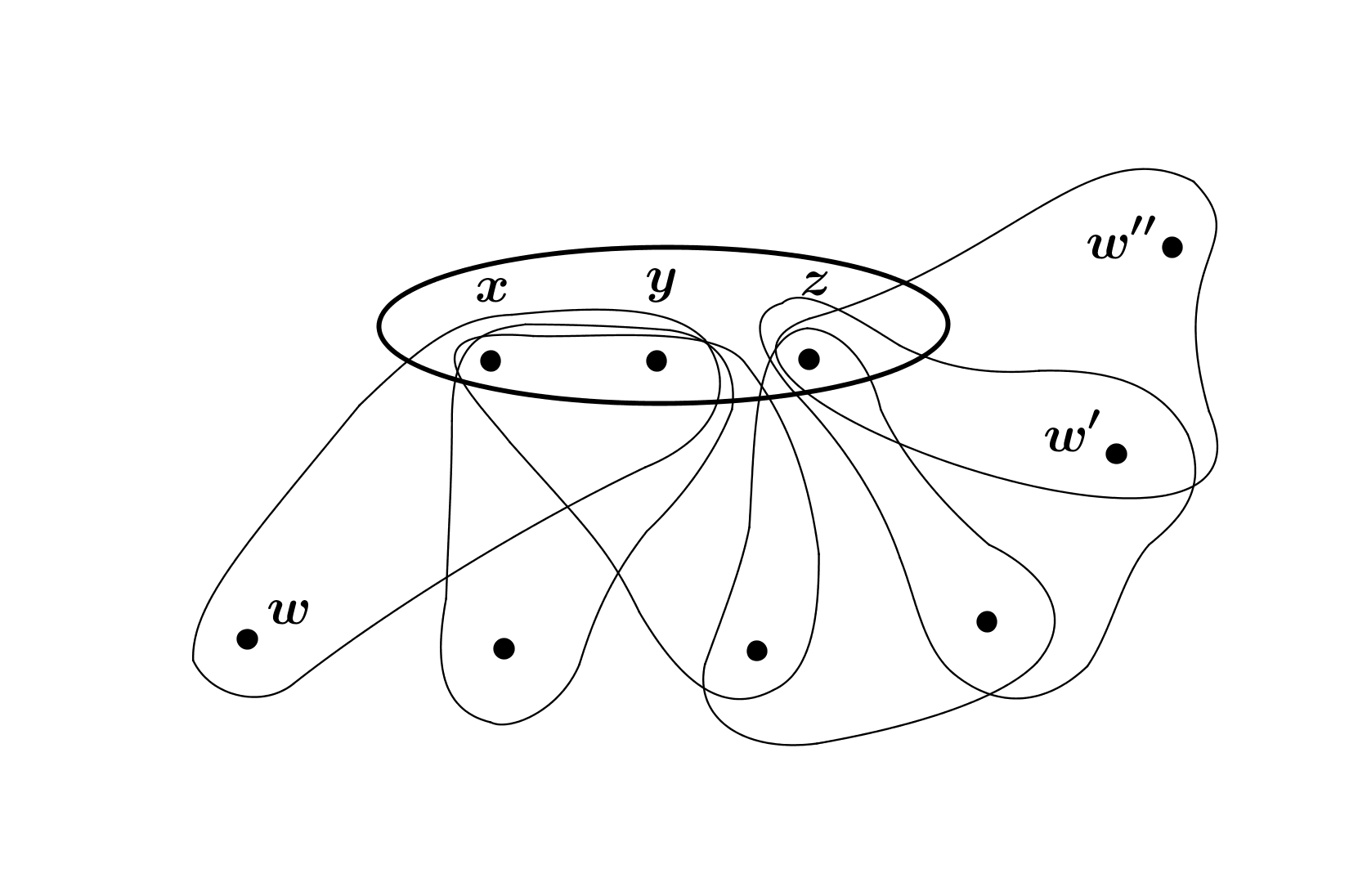}
\caption{ Part of the 3-graph $H(n;C|P)$}
\label{Fig12}
\end{figure}

Note that $P^3_3\subseteq H(n;C|P)\not\supseteq C^3_3$ and thus
$$
ex_3(n;C^3_3|P_3^3)\ge|H(n;C|P)|=1+(n-3)+\binom{n-3}2=\binom{n-2}2+1.
$$
So, again a conditional Tur\'an number, though not yet determined, is going to be not much smaller than its unconditional counterpart.
This is not a coincidence.
In fact, we have the following observation.
\begin{prop}If  $\mathcal F$ consists of connected $k$-graphs only and neither $\mathcal F$ nor $\mathcal G$ depends on $n$,  then
$$ex_k(n;\mathcal F|\mathcal G)\sim ex_k(n;\mathcal F).$$
\end{prop}
\proof By considering a disjoint union of any $G\in\mathcal G$ and
any extremal $\mathcal F$-free graph on $n-|V(G)|$ vertices, we have
$$ ex_k(n-|V(G)|;\mathcal F)+|E(G)|\le ex_k(n;\mathcal F|\mathcal G)\le ex_k(n;\mathcal F).$$
Moreover, by removing $g=|V(G)|$ vertices of smallest degrees from an extremal $\mathcal F$-free $k$-graph on $n$ vertices, we infer that
 $$ex_k(n-g;\mathcal F)\ge ex_k(n;\mathcal F)\left(1-\frac{kg}{n-g}\right).\qed $$

\subsection{Nontrivial intersecting families}

For disconnected $F$, conditioning on the presence of specified subhypergraphs may cause a Tur\'an number drop significantly.
A prime example of this phenomenon is the celebrated Erd\H os-Ko-Rado
Theorem on the maximum size of intersecting families. It asserts that for $n\ge2k$, with $M^k_2$ standing for a pair of disjoint $k$-sets, $ex_k(n;M^k_2)=\binom{n-1}{k-1}$, and, for $n\ge 2k+1$, $Ex_k(n;M^k_2)=\{S^k_n\}$.
It was thus quite natural to ask what is the largest number of edges in an $n$-vertex $M^k_2$-free $k$-graph which is not a star (the so called \emph{nontrivial} intersecting family). Hilton and Milner \cite{hm} proved that the answer to this question is $\binom{n-1}{k-1}- \binom{n-k-1}{k-1} +1$ (see \cite{ff2}
for a short proof).

For $k=3$, it can be checked that an
intersecting triple system is not a star if, and only if, it contains either the triangle $C^3_3$
or the  3-graph $F_5=(\{a,b,c,d,e\},\;\{\{a,b,c\},\{c,d,e\},\{e,a,b\}\})$, or  the clique $K^3_4$.
 From this perspective, the above strengthening of the E-K-R Theorem, due to Hilton and Milner,
  can be reformulated, for $k=3$, as
\begin{equation}\label{HM}
ex_3(n;M^3_2|\{C^3_3,F_5,K^3_4\})=3n-8.
\end{equation}
 Hence, for $\mathcal F=\{M_2^3\}$, a conditional
Tur\'an number can be much smaller than the unconditional one (linear vs. quadratic function of $n$.)

\subsection{Second order Tur\'an numbers}

The Tur\'an numbers for $P^k_3$ and $C^k_3$ reveal a whole lot of similarity  to the E-K-R Theorem. Indeed, restricting just to the case $k=3$, we have, for $n\ge8$,
$$ex_3(n;P^3_3)=ex_3(n;C^3_3)=ex_3(n;M_2^3)=\binom{n-1}2$$
and
$$Ex_3(n;P^3_3)=Ex_3(n;C^3_3)=Ex_3(n;M_2^3)=\{S^3_n\}.$$
 Therefore, like in  the E-K-R case, one might ask for the largest size of a \emph{nontrivial} $P^3_3$-free (or $C^3_3$-free)  3-graph, that is, one which is not a star.

Let us generalize this question. Suppose that for some $n$ and $F$, we have $Ex_k(n;F)=\{H(n;F)\}$, that is, there is a unique (up to isomorphism) extremal $F$-free $n$-vertex $k$-graph $H(n;F)$.
Let $\overline{ex}_k(n;F)$ be the largest number of edges in an $F$-free $n$-vertex $k$-graph $H$ such that $H\not\subseteq  H(n;F)$.
(Besides, the nontrivial intersecting families, a version of this parameter has been studied already for  cliques in graphs,  see \cite{Amin}, where the classical Tur\'an number  $ex_2(n;K_t)$ was restricted to non-$(t-1)$-partite graphs).

Returning to the Tur\'an numbers for $P^3_3$ and $C^3_3$,
observe that for each $F\in\{P^3_3,C^3_3\}$
$$\overline{ex}_3(n;F)=\max\left[ex_3(n;F|M_2^3),\overline{ex}_3(n;\{F,M_2^3\})\right]$$
and $$\overline{ex}_3(n;\{F,M_2^3\})\le \overline{ex}_3(n;M_2^3)\overset{(\ref{HM})}{=}3n-8.$$
Now, consider the following constructions for $n\ge6$. Let $H(n;P|M)$ be the union of a clique $K_4^3$ and a full star $S_{n-3}^3$ whose center is located at one of the vertices of the clique, but which otherwise is vertex-disjoint from the clique (see Fig. \ref{Fig13}). Then $M^3_2\subseteq H(n;P|M)\not\supseteq P^3_3$ and so
$$ex_3(n;P^3_3|M_2^3)\ge|H(n;P|M)|=\binom{n-4}2+4\ge3n-8$$
for $n\ge11$, which, in turn, implies that
$$\overline{ex}_3(n;P^3_3)=ex_3(n;P^3_3|M_2^3).$$

\bigskip
\begin{figure}[!ht]
\centering
\includegraphics [width=9cm]{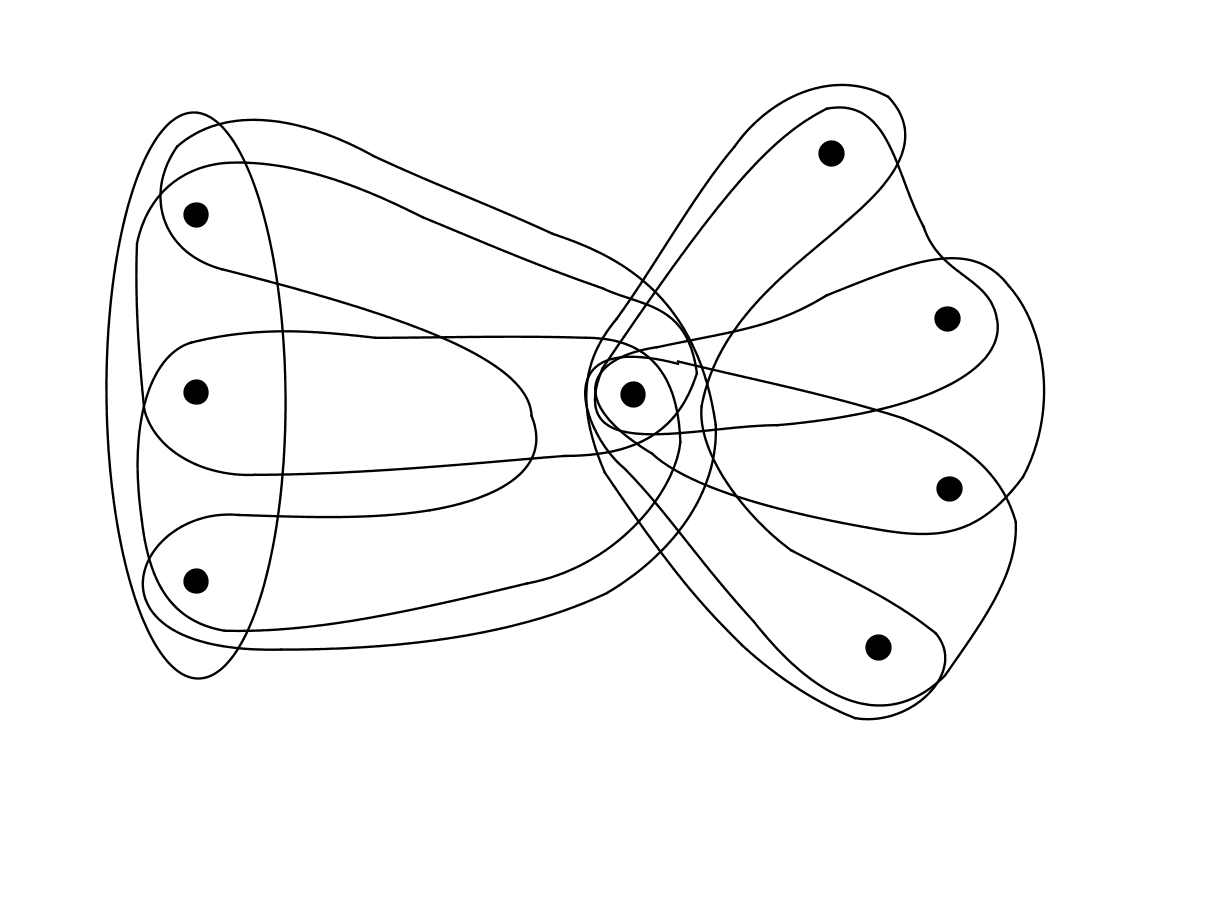}
\caption{ Part of the 3-graph $H(n;P|M)$}
\label{Fig13}
\end{figure}

Moreover, since $M_2^3\subset P_3^3$,
$$ex_3(n;C^3_3|M_2^3)\ge ex_3(n;C^3_3|P^3_3)\ge|H(n;C|P)|\ge\binom{n-2}2+1\ge3n-8$$
for $n\ge8$, and thus, we also have
$$\overline{ex}_3(n;C^3_3)=ex_3(n;C^3_3|M_2^3).$$

\section{Open problems and remarks}

It would be interesting to verify the following conjecture in which we express our belief that these conditional Tur\'an numbers are, indeed, determined by the above described constructions.

\begin{conj} With a possible exception of some small values of $n$,
$$ex_3(n;P^3_3|M_2^3)=\binom{n-4}2+4,$$
$$ex_3(n;C^3_3|M_2^3)=\binom{n-2}2+1.$$
\end{conj}

\begin{remark}\rm
We intend to address the first conjecture  in a forthcoming paper \cite{ejaR}.
If true, it would imply that (again, except for some small $n$)
\begin{equation}\label{equal}
ex_3(n;C^3_3|M_2^3)=ex_3(n;C^3_3|P_3^3).
\end{equation}
 Indeed, if $ex_3(n;P^3_3|M_2^3)\le\binom{n-4}2+4,$ then
$$ex_3(n;C^3_3|P^3_3)\ge|H(n;C|P)|=\binom{n-2}2+1\ge\binom{n-4}2+4\ge ex_3(n;P^3_3|M_2^3).$$
 Thus,
\begin{equation*}\begin{split}ex_3(n;C^3_3|M_2^3)&=\max\left[ex_3(n;C^3_3|\{M_2^3,P^3_3\}),ex_3(n;\{C^3_3,P^3_3\}|M_2^3)\right]\\&\le
\max\left[ex_3(n;C^3_3|P^3_3),ex_3(n;P^3_3|M_2^3)\right]=ex_3(n;C^3_3|P^3_3),\end{split}
\end{equation*}
which, together with the obvious inverse inequality, implies (\ref{equal}).
\end{remark}

\begin{remark}\rm
Conditional Tur\'an numbers defined in this paper may be a useful tool in determining the corresponding Ramsey numbers. For instance, in \cite{j} it has been shown that $R(P^3_3;3)=9$ by observing that if the triples of the clique $K^3_9$ are 3-colored than at least one color appears on more than 28 edges, or all three colors appear each on precisely 28 edges. In either case, Theorem \ref{main} implies that there must be a monochromatic copy of $P^3_3$ (in the latter case, because one cannot partition $K^3_9$ into 3 stars). For more than 3 colors this simple approach does not work any more, but instead one needs to look at the numbers $\overline{ex}_3(n;P^3_3)$ and beyond (see \cite{ejaR}).
\end{remark}


\begin{thebibliography}{HKP99}

\bibitem{Amin} K. Amin, J. Faudree, R. J. Gould, E. Sidorowicz,
\textit{On the non-$(p-1)$-partite $K_p$-free graphs}, Discuss. Math. Graph Theory 33 (2013), 9-23.

\bibitem{bk} N. Bushaw, N. Kettle, \textit{Tur\'an numbers for forests of paths in hypergraphs}, SIAM Journal on Discrete Mathematics 28(2) (2014), 711-721.

\bibitem{ck} R. Cs\'ak\'any, J. Kahn, \textit{A homological Approach to Two Problems on Finite Sets}, Journal of Algebraic
Combinatorics 9 (1999), 141-149.

\bibitem{Erdos} P. Erd\H os, \textit{A problem on independent r-tuples}, Ann. Univ. Sci. Budapest. E\"otv\"os Sect. Math. 8
(1965), 93-95.

\bibitem{ff2} P. Frankl, Z. F\"uredi, \textit{Non-trivial Intersecting Families}, Journal of Combinatorial Theory, Ser. A 41 (1986), 150-153.

\bibitem{ff} P. Frankl, Z. F\"uredi, \textit{Exact solution of some Tur\'an-type problems}, Journal of Combinatorial Theory, Ser. A 45 (1987), 226-262.

\bibitem{fjs} Z. F\"uredi, T. Jiang, R. Seiver, \textit{Exact solution of the hypergraph Tur\'an problem for k-uniform linear
paths}, Combinatorica 34 (3) (2014), 299-322.

\bibitem{hm} A. J. W. Hilton, E. C. Milner, \textit{Some intersection theorems for systems of finite sets}, Quart. J. Math. Oxford Ser. (2) 18 (1967), 369-384.

\bibitem{j} E. Jackowska, \textit{The 3-colored Ramsey number of 3-uniform loose paths of length 3}, submited.

\bibitem{ejaR} E. Jackowska, J. Polcyn, A. Ruci\'nski, in preparation.


\bibitem{kmv} A. Kostochka, D. Mubayi, J. Verstra\"ete, \textit{Tur\'an problems and shadows I: Paths and cycles}, Journal
of Combinatorial Theory, Ser. A, 129 (2015) 57-79.


\end{thebibliography}
\end{document}